\documentclass[reqno]{amsart}

\usepackage[T1]{fontenc}
\usepackage[utf8]{inputenc}
\usepackage[english]{babel}

\usepackage{mathrsfs}
\usepackage{enumerate}

\usepackage{amsmath}
\usepackage{amssymb}
\usepackage{amsthm}

\usepackage{hyperref}

\theoremstyle{plain}
\newtheorem{thm}{Theorem}[section]
\newtheorem{lem}[thm]{Lemma}
\newtheorem{cor}[thm]{Corollary}
\newtheorem{prop}[thm]{Proposition}

\theoremstyle{definition}
\newtheorem{dfn}[thm]{Definition}
\newtheorem{ex}[thm]{Example}
\newtheorem{problem}{Problem}

\theoremstyle{remark}
\newtheorem{rem}[thm]{Remark}

\begin{document}

\title[The global symmetry group and rectification]
{The global symmetry group of first order differential equations and the global rectification theorem}

\author[E.~Gselmann]
{Eszter Gselmann}
\address{Institute of Mathematics, University of Debrecen, 
4010 Debrecen, Pf.\ 12, Hungary}
\email{gselmann@science.unideb.hu}

\author[G.~Horv\'{a}th]
{G\'{a}bor Horv\'{a}th}
\address{Institute of Mathematics, University of Debrecen, 
4010 Debrecen, Pf.\ 12, Hungary}
\email{ghorvath@science.unideb.hu}

\date{July 14, 2019}

\begin{abstract}
Symmetry analysis can provide a suitable change of variables, i.e., 
in geometric terms, a suitable diffeomorphism that simplifies the given direction field, 
which can help significantly in solving or studying differential equations. 
Roughly speaking this is the so-called rectification theorem. 
The local version of this result is a well-known theorem in the field of ordinary differential equations. 
In this note we prove a global counterpart when the equation fulfils the Lipschitz condition. 
Then we use this result to determine the global symmetry group of such an ordinary differential equation. 
It turns out that, 
assuming the Lipschitz condition, 
the full symmetry group is a smooth wreath product of two diffeomorphism groups, 
and does not depend on the form of the equation, at all. 
\end{abstract}


\maketitle

\section{Introduction}\label{sec:intro}

Symmetry analysis is one of the most important tools developed to solve differential equations. 
For example, Olver \cite{Olver} provides a rather comprehensive study on the topic with several examples coming from Physics, 
but even Biology can benefit greatly from the theory (see e.g.\ Nucci \cite{Nucci}). 
Finding some symmetries for a differential equation helps to derive an appropriate change of coordinates which would decrease the number of freedom. 
This is done by rectifying the vector field corresponding to the symmetry. 
By the well-known (local) rectification theorem, 
any nonzero vector field can be rectified in a neighbourhood. 
Such local rectifications, however, can not necessarily been patched together into a global rectification. 

This paper derives a global counterpart of the rectification theorem. 
In particular, 
for an open interval $I$, an open set $M \subseteq \mathbb{R}^n$ and a smooth vector field $v \colon I \times M \to \mathbb{R}^n$, 
we prove that the direction field of the equation 
\begin{equation}\label{eq:diffeq}
 \dot{x}(t) = v(t, x) 
\end{equation}
is globally rectifiable, provided that $v$ is Lipschitzian in its second coordinate, 
and that the solutions of \eqref{eq:diffeq} do not leave the domain of definition. 
The second condition sounds quite restrictive at first glance. 
However, in many situations this condition is automatically satisfied, 
e.g.\ when $v$ is defined on $I \times \mathbb{R}^n$, 
or when an autonomous differential equation is defined on a compact manifold. 
Further, 
for several biological examples (see e.g.\ \cite{Tyson} for the basic function components for biological differential equations) the dependent variables are quantities of different molecules or substrates, 
where one can bound the quantities, 
and one knows a priori that none of the solutions can leave the domain determined by these conditions. 
It is interesting to note that many differential equations motivated by biological examples (e.g.\ all of the ones found in \cite{Tyson}) satisfy the Lipschitz condition, 
as well. 

Nevertheless, 
the global rectification theorem is not the main result of the paper. 
Indeed, 
it is not particularly hard to derive some version of it using well-known properties of time-dependent vector fields and flows. 
The main result is, 
that with the help of this global rectification theorem, 
one can determine the full symmetry group of \eqref{eq:diffeq} up to isomorphism. 
It turns out that the full symmetry group does not depend on the form of the equation, 
but rather on $I$ and $M$. 
In particular, 
we prove that 
(provided the conditions of the global rectification theorem hold)
the full symmetry group of \eqref{eq:diffeq} is the smooth wreath product of the full diffeomorphism groups of $I$ and $M$. 
This underlines the importance of the full diffeomorphism groups of different manifolds (or even of $\mathbb{R}^n$) in yet another way. 

This result by itself of course does not help solving ordinary differential equations. 
Our original motivations to determine the global symmetry group of \eqref{eq:diffeq} were twofold. 
Firstly, 
by knowing the global symmetry group, 
one could hope to recognize some new patterns of properties of the differential equation in question. 
Secondly, 
and somewhat related to the first, 
by determining the global symmetry group of several classes of differential equations, 
one can obtain a (possibly new) characterization of such equations, 
namely according to their group of global symmetries. 
Such characterizations on the local level started with Lie, 
and has been gaining an increasing momentum recently (see e.g.~\cite{LieRemarkable}), 
and may help in understanding the underlying properties of the equations better.

Along this line of thought, 
our result shows another aspect of how differently first order and higher order ordinary differential equations behave. 
One such (well-known) aspect is that first order systems have an infinite dimensional global symmetry group, 
while higher order systems have a finite dimensional global symmetry group. 
This is complemented by the consequence of our result. 
Namely, 
that the global symmetry group of first order systems (satisfying the Lipschitz condition) are isomorphic to each other, 
which is certainly not the case for higher order ordinary systems. 
For example, 
the global symmetry group of the differential equation $\ddot{x} = v(x)$ can be either 1-, or 8-dimensional, 
depending on $v$, 
even if $v$ is assumed to be smooth and Lipschitzian, or even bounded (see e.g.\ \cite{PaliathanasisLeach2016} and references therein). 

The paper is structured as follows. 
In Section~\ref{sec:prelim} we recall some necessary notions from differential geometry and from the theory of differential equations. 
These are used to prove the global version of the rectification theorem (Theorem~\ref{thm:global}) in Section~\ref{sec:rect}. 
Moreover, 
in Section~\ref{sec:rect} we recall some consequences of Theorem~\ref{thm:global}, 
and provide some examples that show that the conditions of Theorem~\ref{thm:global} are necessary. 
Then, in Section~\ref{sec:symmetry} we apply the global rectification theorem to determine the full symmetry group of \eqref{eq:diffeq}. 
The main result of the paper is Corollary~\ref{cor:symmetry}. 
Finally, 
we list some open questions in Section~\ref{sec:problems}.

\section{Preliminaries}\label{sec:prelim}

Here we recall some basic facts from differential geometry, 
and some existence and uniqueness theorems for systems of first order ordinary differential equations. 
These theorems can be found e.g.\ in \cite{Arnold, tDieck, Walter2}. 
The main purpose of this section is to fix the terminology and the notations for the paper.

\begin{dfn}
 Let $M, N\subset \mathbb{R}^{n}$ be  manifolds, a \emph{diffeomorphism} is a one-to-one mapping 
$f\colon M\to N$ such that both $f$ and its inverse are smooth. 
Note, that in this paper by \emph{smooth} we mean a continuously differentiable function
(rather than a function that is infinitely many times differentiable). 
For a manifold $M$, 
we denote by $\mathcal{SYM}(M)$ the group of all diffeomorphisms of $M$ onto itself. 
For an open interval $I \subseteq \mathbb{R}$ we say that a diffeomorphism $\Phi \colon I \times M\to I \times N$ is \emph{time-preserving}, 
if $\Phi \left(t, x\right) = \left( t, \varphi \left(t, x \right) \right)$ for some smooth map $\varphi \colon I \times M \to N$. 
\end{dfn}

\begin{dfn}
 Let $f\colon M\to N$ be a smooth mapping 
from the manifold $M$ into the manifold $N$, and let 
$v$ be a vector attached at the point $x\in M$. 
Then at the point $f(x)\in N$ also arises a vector denoted by $f_{\ast x}v$ called 
\emph{the image of the vector $v$ under the mapping $f$} and it is defined by 
\[
 f_{\ast x}v= \left. \dfrac{d}{dt} f\left(\varphi(t)\right)\right|_{t=0}, 
\]
where 
\[
 \varphi(0)=x \quad \text{and }\quad \left. \dfrac{d}{dt} \varphi(t)\right|_{t=0}=v. 
\]
\end{dfn}
 
\begin{dfn}
 The set of velocity vectors of motions leaving the point $x\in M$ is a linear space, called 
the \emph{tangent space} 
of the manifold 
$M$ at the point $x\in M$ and it is denoted by $T_{x}M$. 
The \emph{tangent bundle} of the manifold $M$ is the disjoint union of the tangent spaces
\[
TM = \bigcup_{x \in M} \left\{\, x \,\right\} \times T_x M. 
\]
\end{dfn}

\begin{prop}
 The mapping $f_{\ast x}\colon T_{x}M\to T_{f(x)}N$ is linear. 
The linear mapping $f_{\ast x}$ is called the \emph{derivative} of the function $f$ at $x\in M$. 
\end{prop}

\begin{dfn}\label{dfn:f_ast}
We say that on the 
manifold 
$M$ \emph{a (smooth) vector field is given} if for all $x\in M$ a vector 
$v(x)\in T_{x}M$ is assigned such that the dependence from $x$ is smooth. 
That is, 
if $v \colon M \to TM$ is a smooth function. 
Assume that a smooth vector field $v$ is given on the 
manifold
$M$, further let 
$f\colon M\to N$ be a diffeomorphism. Then the \emph{image of the vector field $v$ under the diffeomorphism $f$} is defined by
\[
 w(y)= \left(f_{\ast x}\right)v(x) 
\qquad 
\left(y\in N\right), 
\]
where $y=f(x)$. 
\end{dfn}

\begin{thm}\label{thm:equivdiffeq}
 Let $f\colon M\to N$ be a diffeomorphism. Then  the differential equation 
\begin{equation}\label{eq:rect1}
\dot{x}=v(x) 
\qquad 
\left(x\in M\right)
\end{equation}
with phase space $M$ determined by the vector field $v$ is equivalent to the equation 
\begin{equation}\label{eq:rect2}
\dot{y}=\left(f_{\ast x}v\right)(y)
\qquad 
\left(y = f(x)\in N\right)
\end{equation}
with phase space $N$ determined by the vector field $f_{\ast }v$. 
More precisely, a function 
$\varphi\colon I\to M$ is a solution to equation~\eqref{eq:rect1} is and only if 
$f\circ \varphi\colon I\to N$ is a solution to equation~\eqref{eq:rect2}. 
\end{thm}

Now, we continue by briefly reviewing some basic facts from the theory of ordinary differential equations and from real analysis. 
We begin with a generalization of the celebrated result of Picard and Lindel\"{o}f, for further results consult e.g.\ Walter \cite{Walter2}.

\begin{thm}[Picard\,--\,Lindel\"{o}f]\label{thm:PicardLindelof}
 Let $I\subset \mathbb{R}$ be a nonvoid interval, 
 $M \subseteq \mathbb{R}^n$ be an open set, 
$v\colon I\times M \to\mathbb{R}^{n}$ be a continuous function and let us assume that there exists a 
continuous function  
$L\colon I\to [0, +\infty[$ such that 
\[
\tag{$\mathscr{L}$}\label{eq:L}
\left\|v(t, x_{1})-v(t, x_{2})\right\|\leq L(t)\left\|x_{1}-x_{2}\right\| 
\quad 
\left(t\in I, x_{1}, x_{2}\in M \right). 
\]
Let $t_0\in I$, $x_0\in M$, then the Cauchy problem 
\begin{equation}\label{eq:Cauchy}
\left\{
\begin{array}{rcl}
 \dot{x}(t)&=&v(t, x) \\
x(t_0)&=&x_0
\end{array}
\right.
\end{equation}
admits exactly one 
$\varphi\colon I_0\to M$ solution for some interval $I_0 \subseteq I$ containing $t_0$, 
and any other solution is a restriction of this solution. 
Further, if the solutions cannot leave the domain $I\times M$ for some reason 
(e.g.\ when $M = \mathbb{R}^n$), 
then  $I_0=I$ holds, as well, 
i.e.\ in such a case the (unique maximal) solution is defined on the whole $I$. 
\end{thm}

During the proof of the rectification theorem, a `continuous dependence on the inital data' 
type result will also be utilized, c.f.\ Arnol'd \cite{Arnold}.

\begin{thm}[$\mathscr{C}^{r}$ dependence]\label{Thm8}
 Let $r\geq 1$ be arbitrarily fixed, $I\subset \mathbb{R}$ be a nonvoid interval, $t_0\in I$ be fixed, 
 $M$ be a manifold, 
$v\colon I\times M\to\mathbb{R}^{n}$ be a function which fulfils condition~\eqref{eq:L} and assume 
that $v\in \mathscr{C}^{r}(I\times M)$. 
Denote by $\varphi(t, x_0)$ the uniquely determined solution of the Cauchy problem \eqref{eq:Cauchy}
for all $x_0\in M$. 
Then the function 
\[
 I \times M \ni (t, x_0) \longmapsto \varphi(t, x_0)
\]
is $r$ times continuously differentiable. 
\end{thm}

Further, we will need the Brouwer homeomorphism (also known as the invariance of the domain) theorem, 
see e.g.\ tom Dieck \cite{tDieck}. 

\begin{thm}[Brouwer]\label{thm:Brouwer}
 Let $M\subset \mathbb{R}^{n}$ be a nonempty, open set and 
$f\colon M\to \mathbb{R}^{n}$ be an injective and continuous mapping. 
Then the set $f(M)\subset \mathbb{R}^{n}$ is open, further, 
$f$ is a homeomorphism from $M$ onto $f(M)$. 
\end{thm}

\section{The global rectification theorem}\label{sec:rect}

\begin{dfn}
A \emph{rectification} of a direction field is a diffeomorphism mapping the direction field into a field of parallel directions. 
A field is said to be \emph{rectifiable} if there exists a rectification for it. 
\end{dfn}

A direction field of parallel directions is easily integrable. 
Thus, if a direction field is rectifiable, 
then (as every first order ordinary differential equation is equivalent to an autonomous one)
one can apply Theorem~\ref{thm:equivdiffeq} to solve the original equation. 

\begin{thm}[Global rectification]\label{thm:global}
 Let $I\subset \mathbb{R}$ be a nonempty open interval and 
$M\subset \mathbb{R}^{n}$ be an open set, 
$v \colon I\times M \to \mathbb{R}^n$ be an $r$ times continuously differentiable function for some $r \geq 1$ such that condition~\eqref{eq:L} is fulfilled. 
Further, 
assume that for every $t_0 \in I$, $x_0 \in M$ the unique maximal solution of the Cauchy problem \eqref{eq:Cauchy},
guaranteed by the Picard\,--\,Lindel\"{o}f theorem (Theorem~\ref{thm:PicardLindelof}), 
cannot leave the domain $I \times M$. 
Then the direction field of the equation 
\[
 \dot{x}=v(t, x)
\]
is globally rectifiable by a time-preserving diffeomorphism from the class $\mathscr{C}^{r}$. 
In particular, 
there exists a $\mathscr{C}^{r}$-diffeomorphism $\Phi \colon I \times M \to I \times M$ such that 
\[
\left( \Phi^{-1} \right)_{\ast (t, x)} v(t, x)  = 1\dfrac{\partial }{\partial t}+0\dfrac{\partial}{\partial x}.
\]
\end{thm}


\begin{proof}
Fix $t_{0}\in I$. Further, 
let $x_0\in M$ be arbitrarily fixed and let us consider the Cauchy problem \eqref{eq:Cauchy}
\[
\left\{
\begin{array}{rcl}
 \dot{x}(t)&=&v(t, x)\\
x(t_{0})&=&x_0. 
\end{array}
\right.
\]
Due to the theorem of Picard and Lindel\"{o}f (Theorem~\ref{thm:PicardLindelof}), there exists a uniquely determined solution to this problem, say, 
$\varphi(t, x_0)$ on the whole $I$. Furthermore, the mapping 
\[
 M\ni x_0\longmapsto \varphi(t, x_0)
\]
does exist and it depends on $x_0$ in a continuously differentiable way. 
Additionally, due to 
Theorem \ref{Thm8}, we also have $\varphi\in \mathscr{C}^{r}(I\times M)$. 
Now, let us consider the mapping $\Phi\colon I\times M\to I\times \mathbb{R}^{n}$ defined by 
\[
 \Phi(t, x_0)= \left(t, \varphi(t, x_0)\right) 
\qquad 
\left((t, x_0)\in I\times M\right). 
\]
We will show that $\Phi$ is bijective on $I \times M$, and 
its inverse
is the desired rectification of the direction field. 
In particular, we have that
\begin{enumerate}[(i)]
\item\label{it:rect0} 
the mapping $\Phi$ is surjective onto $I \times M$, 
  \item\label{it:recti} $\Phi$ is differentiable, and it  even belongs to the class 
$\mathscr{C}^{r}(I\times M)$, 
\item\label{it:rectii} $\Phi_{\ast x}$ maps the vector field
$1\dfrac{\partial }{\partial t}+0\dfrac{\partial}{\partial x}$ into the vector field 
$1\dfrac{\partial }{\partial t}+v\dfrac{\partial }{\partial x}$, 
\item\label{it:rectiii} $\Phi$ is a diffeomorphism on $I\times M$. 
\end{enumerate}

First, we prove (\ref{it:rect0}). 
Let $(\tau, \xi) \in I \times M$ be arbitrary, 
and consider the Cauchy problem 
\[
\left\{
\begin{array}{rcl}
 \dot{x}(t)&=&v(t, x)\\
x(\tau)&=&\xi. 
\end{array}
\right.
\]
By the Picard\,--\,Lindel\"{o}f theorem (Theorem~\ref{thm:PicardLindelof}), 
this has a uniqe maximal solution $\varphi$ on the whole $I$. 
In particular, 
$\varphi$ attains some value $x_0$ at $t_0$ (i.e.\ $\varphi(t_0)=x_0$). 
Thus the Cauchy problem 
\[
\left\{
\begin{array}{rcl}
 \dot{x}(t)&=&v(t, x)\\
x(t_0)&=&x_0. 
\end{array}
\right.
\]
has the unique solution $\varphi$, as well. 
Therefore $\Phi \left( \tau, \varphi \left( \tau, x_0 \right) \right) = \left( \tau, \xi \right)$. 

Now, (\ref{it:recti}) follows immediately from Theorem \ref{Thm8}. 
Further, 
every obtained $\varphi$ is a solution of our ordinary diffe\-rential equation, 
hence (\ref{it:rectii}) clearly holds. 

Finally, we prove (\ref{it:rectiii}).
By (\ref{it:recti}) we have $\Phi\in \mathscr{C}^{r}(I\times M)$. 
Due to the uniqueness part of the Picard\,--\,Lindel\"{o}f theorem (Theorem~\ref{thm:PicardLindelof}), 
$\Phi$ is injective. 
From (\ref{it:rect0}) we have that $\Phi$ is surjective onto $I \times M$. 
Due to the theorem of Brouwer (Theorem~\ref{thm:Brouwer}), 
$\Phi$ is a differentiable homeomorphism. 
Since $\Phi_{\ast x}$ is non-vanishing, 
$\Phi$ is a diffeomorphism. 
\end{proof}

In case the continuous function $v$ does not fulfil condition~\eqref{eq:L}, 
we can only expect local and in general nonunique solvability. 
This might suggest that in such a case only the local version of the rectification theorem is valid, 
which follows easily from its global counterpart. 

\begin{cor}[Local rectification]\label{cor:local}
Every smooth direction field is rectifiable in a neighbourhood of each of its points. 
If the field is $r$ times continuously differentiable, then the rectifying diffeomorphism can also be taken from the class $\mathscr{C}^{r}$. 
\end{cor}

\begin{proof}
 Let us assume that the differential equation corresponding to the direction field is 
\[
 \dot{x}=v(t, x) 
\qquad 
\left(t \in I, x\in M\right), 
\]
where $I\subset \mathbb{R}$ is a nonempty open interval, $M\subset \mathbb{R}^{n}$ is an open set and 
$v\in \mathscr{C}^{r}(I\times M)$. 

Let $(t_{0}, x_{0})\in I\times M$ be arbitrarily fixed. Then there exists a compact and 
convex subset $J\times C\subset I\times M$ such that 
$\left. v\right|_{J\times C}$ is a Lipschitz function, with Lipschitz constant  
$ L= \sup_{
J\times C}\left\|v_{\ast x} \right\|$. 
Then 
$ L\geq \sup_{
J^{\circ}\times C^{\circ}}\left\|v_{\ast x} \right\|$, 
where $S^{\circ}$ denotes the interior of the set $S$. This shows that 
the restriction of $v$ to the open set $J^{\circ}\times C^{\circ}$ is a Lipschitz function and 
$v$ is continuous on this set, too. Thus, the assumptions of the Global rectification theorem (Theorem~\ref{thm:global}) are fulfilled. 
Therefore the direction field on $J^{\circ}\times C^{\circ}$ is rectifiable by a $\mathscr{C}^{r}$-diffeomorphism. Since 
$(t_{0}, x_{0})\in I\times M$ was arbitrary, we obtain that every point of the extended direction field 
admits a neighbourhood in which the direction field is rectifiable. 
\end{proof}

It is well-known that more regularity of the right hand side of the 
differential equation \emph{does not necessarily impose} existence and uniqueness, even if the 
right hand side of the equation is analytic. 

\begin{ex}
Consider the Cauchy problem 
\begin{equation}\label{eq:Cauchy2}
\left\{
 \begin{array}{rcl}
  \dot{x}(t)&=&x^{2}(t)\\
x(0)&=&x_{0}. 
\end{array}
\right.
\qquad 
\left(t\in \mathbb{R}\right). 
\end{equation}
Here, the mapping $v\colon \mathbb{R}\times \mathbb{R}\to \mathbb{R}$
defined by 
$ v(t, x)= x^{2} $
is of the class $\mathscr{C}^{\omega}(\mathbb{R}\times \mathbb{R})$, 
but 
\eqref{eq:Cauchy2} has a unique maximal solution only on 
$]-\infty, x_{0}^{-1}[$. 
Namely, 
$ x(t)= \dfrac{1}{\frac{1}{x_{0}}-t}$, 
$t\in \left]-\infty, x_{0}^{-1}\right[$. 
Moreover, 
\eqref{eq:Cauchy2} does not have a solution that is defined on the whole line $\mathbb{R}$. 
Here, the right hand side of the differential equation does not fulfil condition 
\eqref{eq:L}, but it is continuously differentiable. 
Thus, as this example shows, there are differential equations whose direction field is 
locally rectifiable at every point, but globally not. 
\end{ex}

\begin{ex}
Let us 
consider the Cauchy problem 
\begin{equation}\label{eq:Cauchy3}
\left\{
 \begin{array}{rcl}
  \dot{x}(t)&=&2\sqrt{\left| x(t)\right|}\\
x(0)&=&0. 
 \end{array}
\right.
\end{equation}
It is easy to see that for any $\alpha\geq 0$ the function 
\[
 x_{\alpha}(t)=
\begin{cases}
0, & t\leq \alpha\\
(t-\alpha)^{2}, & t\geq \alpha 
\end{cases}
\]
solves the initial value problem ~\eqref{eq:Cauchy3}. 
Thus, this Cauchy problem does not admit a unique solution, 
due to the fact that the mapping $(t,x) \longmapsto 2\sqrt{| x |}$
does not fulfil condition \eqref{eq:L}. 
Therefore, the direction field of this differential equation is not globally rectifiable. 
However, every point $(x, y)$ of the plane for which $x^{2}+y^{2}\neq 0$ does admit a neighbourhood in which the 
direction field of the equation is rectifiable. 
\end{ex}

\begin{rem}
Note, that even condition \eqref{eq:L} can be weakened. 
By Osgood \cite{Osgood}, 
it is enough to assume that 
\[
\tag{$\mathscr{L'}$}\label{eq:L'}
\left\|v(t, x_{1})-v(t, x_{2})\right\|\leq L(t)\omega \left( \left\|x_{1}-x_{2}\right\| \right)
\quad 
\left(t\in I, x_{1}, x_{2}\in M \right), 
\]
where $\omega \colon \left[ 0, +\infty \right[ \to \left[ 0, +\infty \right[$ is a nondecreasing function, 
$\omega (0) =0$, $\omega(\xi) >0$ for $\xi >0$ and satisfies \emph{Osgood's condition}, 
that is 
\[
\int_0^1 \frac{d \xi}{\omega(\xi)} = \infty. 
\]
However, 
in most of the situations condition \eqref{eq:L} is enough. 
\end{rem}

\section{The global symmetry group of first order systems}\label{sec:symmetry}

Using \cite{Olver}, 
we recall the definition of the action of a smooth map on a function, 
and then define the global version of the Lie symmetry group of \eqref{eq:diffeq} by requiring all transformations having full domain. 
These definitions can be extended to a more general setting, 
but for the purposes of our paper, they suffice.

\begin{dfn}[{\cite[Section~2.2]{Olver}}]
Let $I\subset \mathbb{R}$ be a nonempty open interval and 
$M\subset \mathbb{R}^{n}$ be an open set. 
Let $\alpha \colon I \times M \to I \times M$ be a smooth map, 
and let $\varphi \colon J \to M$ be a smooth function for some $J \subseteq I$. 
Let $\Gamma_{\varphi} = \left\{\left( t, \varphi \left( t \right) \right) : t \in J \right\}$ denote the graph of $\varphi$. 
If the transform of $\Gamma_{\varphi}$ by $\alpha$ 
(i.e.\ $\alpha \circ \Gamma_{\varphi} = \left\{\alpha \left(  t, \varphi \left( t  \right) \right) : t \in J \right\}$) is a graph of a function $\tilde\varphi$, 
then we write $\tilde\varphi = \alpha \circ \varphi$, 
we say $\alpha \circ \varphi$ \emph{is defined}, 
and call $\tilde\varphi$ the \emph{transform} of $\varphi$ by $\alpha$. 
\end{dfn}

\begin{dfn}[Global version of {\cite[Definition~2.23]{Olver}}]
Let $I\subset \mathbb{R}$ be a nonempty open interval and 
$M\subset \mathbb{R}^{n}$ be an open set, 
$v \colon I\times M \to \mathbb{R}^n$ be a 
continuously differentiable function. 
We say that a bijective, smooth map $\alpha \colon I \times M \to I \times M$ is a \emph{global symmetry} of the equation \eqref{eq:diffeq} if whenever $x = \varphi (t)$ is a solution of \eqref{eq:diffeq} and $\alpha \circ \varphi$ is defined, 
then $\alpha\circ \varphi$ is also a solution of \eqref{eq:diffeq}. 
(By \emph{solution} we mean any smooth function $\varphi \colon J \to M$ for some open interval $J \subseteq I$.)
The \emph{global symmetry group} of \eqref{eq:diffeq} is the 
group consisting of all bijective, smooth transformations of $I \times M$ that are symmetries of \eqref{eq:diffeq}. 
\end{dfn}

By \cite[Definition~1.25]{Olver}, 
a \emph{global group of transformations} of $I \times M$ is a subgroup of $\mathcal{SYM}(I \times M)$.
That is, 
the global symmetry group is the unique largest global group of transformations acting on $I \times M$ moving solutions of \eqref{eq:diffeq} to solutions of \eqref{eq:diffeq}. 


\begin{thm}\label{thm:symmetry}
Let $I\subset \mathbb{R}$ be a nonempty open interval and 
$M\subset \mathbb{R}^{n}$ be an open set, 
$v \colon I\times M \to \mathbb{R}^n$ be a continuously differentiable function such that the conditions of Theorem~\ref{thm:global} are fulfilled. 
Then the global symmetry group $\mathcal{G}$ of the equation~\eqref{eq:diffeq} 
is isomorphic to the group
\[
\left\{\, h \in \mathcal{SYM} (I \times M) \mid h(t, x) =(f(t, x), g(x)), t \in I, x \in M \,\right\}, 
\]
where the group multiplication is function composition. 
\end{thm}

For proving Theorem~\ref{thm:symmetry}, 
we first prove that any two differential equations satisfying the conditions of Theorem~\ref{thm:global} have isomorphic symmetry groups. 
Then we determine the symmetry group of the trivial equation $\dot{x}=0$. 

\begin{lem}\label{lem:symmetry}
 Let $I \subset \mathbb{R}$ be a nonempty open interval and 
$M_1, M_2\subset \mathbb{R}^{n}$ be open sets, 
$v_j \colon I \times M_j \rightarrow \mathbb{R}^n$ ($j \in \left\{\, 1, 2 \, \right\}$) be continuously differentiable functions for which 
the direction fields of the equations
\[
 \dot{x}=v_1(t, x) 
 \quad 
 \text{and} 
 \quad 
 \dot{x}=v_2(t, x)
\]
are diffeomorphic by a time-preserving diffeomorphism $\Phi \colon I \times M_1 \rightarrow I \times M_2$. 
If the global symmetry group of $ \dot{x}=v_j(t, x) $ is denoted by $\mathcal{G}_j$ ($j \in \left\{\, 1, 2 \, \right\}$), 
then 
\[
\mathcal{G}_2 = \left\{\, \Phi \circ \alpha \circ \Phi^{-1} \mid \alpha \in \mathcal{G}_1 \, \right\}.
\]
\end{lem}

\begin{proof}
Let $\alpha \colon I \times M_1 \rightarrow I \times M_1$ be a symmetry of equation 
\[
 \dot{x}=v_1(t, x). 
\]
We will prove that $\Phi \circ \alpha \circ \Phi^{-1} \colon I \times M_2 \rightarrow I \times M_2$ is a symmetry of the equation 
\[
 \dot{x}=v_2(t, x). 
\]
Indeed, 
assume that the function $\varphi \colon I \to M_2$ is a solution of equation $\dot{x} = v_2(t, x)$, 
that is $\varphi'(t) = v_2\left(t, \varphi(t)\right)$ for every $t \in I$. 
Now, $t \mapsto \Phi \circ \alpha \circ \Phi^{-1} \circ \left(t, \varphi(t) \right)$ is a solution of $\dot{x} = v_2(t, x)$, 
since the second coordinate of its derivative is 
\begin{multline*}
 \left( \left( \Phi \circ \alpha \circ \Phi^{-1} \left( t, \varphi(t)\right) \right)' \right)_2 = 
\left( \Phi \circ \alpha \circ \Phi^{-1} \right)_{\ast (t,\varphi(t))} v_2\left( t, \varphi(t) \right) = \\
 \left( \Phi \circ \alpha \right)_{\ast \Phi^{-1}(t,\varphi(t))} \Phi^{-1}_{\ast \left( t, \varphi(t)  \right)} v_2\left( t, \varphi(t)  \right) 
= \left( \Phi \circ \alpha \right)_{\ast \Phi^{-1}(t,\varphi(t))} v_1\left( \Phi^{-1}\left( t, \varphi(t) \right) \right) \\
= \Phi _{\ast \alpha \circ \Phi^{-1}(t,\varphi(t))} \alpha_{\ast \Phi^{-1}\left( t, \varphi(t) \right)} v_1\left( \Phi^{-1}\left( t, \varphi(t) \right) \right) \\
= \Phi _{\ast \alpha \circ \Phi^{-1}(t,\varphi(t))} v_1\left( \alpha \circ \Phi^{-1}\left( t, \varphi(t) \right) \right) 
= v_2 \left( \Phi \circ \alpha \circ \Phi^{-1} \left( t, \varphi (t) \right) \right). 
\end{multline*}
Here, we applied the chain rule and Definition~\ref{dfn:f_ast} several times. 
Similarly, 
one can prove that if $\beta \colon I \times M_2 \rightarrow I \times M_2$ is a symmetry of equation 
\[
 \dot{x}=v_2(t, x), 
\]
then $\Phi^{-1} \circ \beta \circ \Phi \colon I \times M_1 \rightarrow I \times M_1$ is a symmetry of equation 
\[
 \dot{x}=v_1(t, x). 
\]
Thus, the symmetry groups of the equations 
\[
 \dot{x}=v_1(t, x) \quad \text{and} \quad 
 \dot{x}=v_2(t, x)
\]
are isomorphic, 
and they are conjugate by $\Phi$. 
\end{proof}

\begin{proof}[Proof of Theorem~\ref{thm:symmetry}]
By the Global rectification theorem (Theorem~\ref{thm:global}) there exists a time-preserving diffeomorphism $\Phi \colon I \times M \to I \times M$ sending the direction field 
$1\dfrac{\partial }{\partial t}+v\dfrac{\partial }{\partial x}$ into the direction field
$1\dfrac{\partial }{\partial t}+0\dfrac{\partial}{\partial x}$. 
Therefore by Lemma~\ref{lem:symmetry} we have that the symmetry group of equation~\eqref{eq:diffeq} is isomorphic to the symmetry group of the trivial equation
\begin{equation}\label{eq:trivial}
\dot{x} = 0. 
\end{equation}
The solutions of \eqref{eq:trivial} are the horizontal lines $\varphi(t) = c$ for some constant $c \in M$. 
Thus, $\alpha \colon I \times M \to I \times M$ is a symmetry of \eqref{eq:trivial} if and only if it preserves these horizontal lines, 
that is if and only if the second coordinate of $\alpha(t, x)$ depends only on $x$ and not on $t$. 
Thus, 
on the one hand, 
every global symmetry $\alpha \colon I \times M \to I \times M$ has the form $\alpha(t, x) =\left( f(t, x), g(x) \right)$. 
On the other hand, 
every function $h \in \mathcal{SYM} (I \times M)$ such that $h$ is of the form $h(t, x) =\left( f(t, x), g(x) \right)$ transfers solutions to solutions, 
since any (open) horizontal line is moved to another (open) horizontal line. 
Thus, 
the global symmetry group of the trivial equation is 
$\left\{\, h \in \mathcal{SYM} (I \times M) \mid h(t, x) =\left( f(t, x), g(x) \right),  t \in I, x \in M \,\right\}$, 
concluding the proof. 
\end{proof}

Recall the definition of the (smooth) wreath product (see e.g.\ \cite{MeldrumWreath}). 

\begin{dfn}
Let $A$ and $B$ be two transformation groups, 
where $A$ acts on $I$ and $B$ acts on $M$. 
Let $F$ contain all $M \to A$ functions, 
then $F$ is a group with pointwise multiplication. 
Now, $B$ acts naturally on $F$ in the following way: 
for arbitrary $g \in B$, $f \in F$ we have $f^g(x) = f \left(g^{-1}(x)\right)$. 
This action is in fact a $B \to \mathrm{Aut}(F)$ homomorphism. 
Let $W$ be the semidirect product $F \rtimes B$ with this action, 
that is 
\[
\left( f_1, g_1 \right) \left( f_2, g_2 \right) = \left( f_1^{g_2^{-1}}\circ f_2, g_1 \circ g_2 \right), \qquad (f_i, g_i) \in F \rtimes B. 
\]
Then $W$  is called the \emph{wreath product} of $A$ and $B$, 
and is denoted by $A \wr_M B$. 
The elements of $A \wr_M B$ act naturally on $I \times M$ by 
\[
(f, g) (t, x) = \left(f(x)(t) , g(x) \right), \qquad t \in I, x \in M. 
\]
When $A$ and $B$ are diffeomorphism groups, 
acting in a continuously differentiable manner on $I$ and $M$, 
and one takes $F$ only to contain the smooth $M \to A$ functions in this definition, 
then we call the resulting group the \emph{smooth wreath product}, 
and we denote it by $A \wr_{\mathcal{C}^1} B$. 
\end{dfn}

\begin{cor}\label{cor:symmetry}
 Let $I\subset \mathbb{R}$ be a nonempty open interval and 
$M\subset \mathbb{R}^{n}$ be an open set, 
$v \colon I\times M \to \mathbb{R}^n$ be a continuously differentiable function such that the conditions of Theorem~\ref{thm:global} are fulfilled. 
Then the global symmetry group $\mathcal{G}$ of the equation~\eqref{eq:diffeq} 
is isomorphic to the smooth wreath product of all diffeomorphisms of $I$ and all diffeomorphisms of $M$, that is 
\[
\mathcal{G} \simeq \mathcal{SYM} (I) \wr_{\mathcal{C}^1} \mathcal{SYM} (M). 
\]
\end{cor}

\begin{proof}
Let $A = \mathcal{SYM}(I)$, $B = \mathcal{SYM}(M)$, 
then $F = \left\{\, f \colon M \to A \mid f \in \mathcal{C}^1 \, \right\}$. 
The elements of $\mathcal{SYM} (I) \wr_{\mathcal{C}^1} \mathcal{SYM} (M)$ are pairs $(f, g)$ such that $f \in F$, $g \in B$. 
Let
\begin{align*}
\Psi \colon \mathcal{SYM} (I) \wr_{\mathcal{C}^1} \mathcal{SYM} (M) &\to \mathcal{SYM}(I \times M), \\
(f, g) &\mapsto h, 
\end{align*}
such that 
\[
h(t,x) = \left( f(x)(t), g(x)  \right). 
\]
Now, $\Psi$ is clearly injective, 
and is surjective onto the group determined in Theorem~\ref{thm:symmetry}. 
Thus, we only need to prove that $\Psi$ is a homomorphism. 
Let $f_1,f_2 \in F, g_1, g_2 \in \mathcal{SYM}(M)$ be arbitrary, 
and denote $\Psi (f_i, g_i)$ by $h_i$. 
Then 
\begin{multline*}
\Psi\left( f_1, g_1 \right) \circ \Psi \left( f_2, g_2 \right) (t, x) =  h_1 \circ h_2 \left(t, x \right) = h_1 \left( f_2(x)(t), g_2(x) \right) \\
= \left( f_1 \left(g_2(x) \right) \left(f_2 (x)(t) \right), g_1\left( g_2 \left( x \right) \right) \right) \\
= \Psi \left( f_1^{g_2^{-1}} \circ f_2, g_1 \circ g_2 \right) (t, x) = \Psi \left( (f_1, g_1) (f_2, g_2) \right) (t, x). 
\end{multline*}
\end{proof}

\section{Open problems}\label{sec:problems}

Several open questions remain. 
The most obvious one is probably how one can weaken the conditions of the Global rectification theorem (Theorem~\ref{thm:global}) so that its conclusion still holds. 
One way to obtain a more general global rectification is if one can prove a more general version of the Picard\,--\,Lindel\"{o}f theorem. 
However, considering that there are many possibilities to rectify a direction field, 
there may be different approaches. 

\begin{problem}
Give necessary and sufficient conditions, 
or at least weaker sufficient conditions under which Theorem~\ref{thm:global} holds. 
\end{problem}

Another natural problem is the study of the obtained global symmetry group $\mathcal{SYM}(I) \wr_{\mathcal{C}^1} \mathcal{SYM}(M)$. 
For example, 
it would be interesting to know what finite dimensional (or even finite) subgroups this group has. 

\begin{problem}
For an interval $I$ and a manifold $M$ what are the finite dimensional subgroups of $\mathcal{SYM}(I) \wr_{\mathcal{C}^1} \mathcal{SYM}(M)$? 
\end{problem}

\begin{problem}
For an interval $I$ and a manifold $M$ what are the finite subgroups of $\mathcal{SYM}(I) \wr_{\mathcal{C}^1} \mathcal{SYM}(M)$? 
\end{problem}

\section{Acknowledgements}
The research was partially supported by the European Union's Seventh Framework Programme (FP7/2007-2013) under grant agreement no.~318202. 
The first author was partially supported by the Hungarian Scientific Research Fund (OTKA) Grant no.~K111651.
The second author was partially supported by the János Bolyai Research Scholarship of the Hungarian Academy of Sciences, 
by the Hungarian Scientific Research Fund (OTKA) grant no.~K109185, 
and by the National Research, Development and Innovation Fund of Hungary, 
financed under the FK~124814 funding scheme.


\begin{thebibliography}{10}

\bibitem{Arnold}
Vladimir~I. Arnol{'}d, \emph{Ordinary differential equations}, Springer
  Textbook, Springer-Verlag, Berlin, 1992, Translated from the third Russian
  edition by Roger Cooke. \MR{1162307 (93b:34001)}

\bibitem{LieRemarkable}
Gianni Manno, Francesco Oliveri, Giuseppe Saccomandi, and Raffaele Vitolo,
  \emph{Ordinary differential equations described by their {L}ie symmetry
  algebra}, J. Geom. Phys. \textbf{85} (2014), 2--15. \MR{3253544}

\bibitem{MeldrumWreath}
J.~D.~P. Meldrum, \emph{Wreath products of groups and semigroups}, Pitman
  Monographs and Surveys in Pure and Applied Mathematics, vol.~74, Longman,
  Harlow, 1995. \MR{1379113 (97j:20030)}

\bibitem{Nucci}
M.~C. Nucci, \emph{The role of symmetries in solving differential equations},
  Mathl. Comput. Modelling \textbf{25} (1997), no.~8--9, 181--193.

\bibitem{Olver}
Peter~J. Olver, \emph{Applications of {L}ie groups to differential equations},
  second ed., Graduate Texts in Mathematics, vol. 107, Springer-Verlag, New
  York, 1993. \MR{1240056 (94g:58260)}

\bibitem{Osgood}
W.~F. Osgood, \emph{Beweis der {E}xistenz einer {L}{\"{o}}sung der
  {D}ifferentialgleichung {$\frac{{dy}}{{dx}} = f\left( {x,y} \right)$} ohne
  {H}inzunahme der {C}auchy-{L}ipschitz'schen {B}edingung}, Monatsh. Math.
  Phys. \textbf{9} (1898), no.~1, 331--345. \MR{1546565}

\bibitem{PaliathanasisLeach2016}
A.~Paliathanasis and P.~G.~L. Leach, \emph{Comment on ``{C}lassification of
  {L}ie point symmetries for quadratic {L}i\'enard type equation {$\ddot
  x+f(x)\dot x{}^2+g(x)=0$}'' [{J}. {M}ath. {P}hys. 54, 053506 (2013)] and its
  erratum [{J}. {M}ath. {P}hys. 55, 059901 (2014)]}, J. Math. Phys. \textbf{57}
  (2016), no.~2, 024101, 2. \MR{3452695}

\bibitem{tDieck}
Tammo tom Dieck, \emph{Algebraic topology}, EMS Textbooks in Mathematics,
  European Mathematical Society (EMS), Z{\"{u}}rich, 2008. \MR{2456045
  (2009f:55001)}

\bibitem{Tyson}
John~J. Tyson, Katherine~C. Chen, and B{\'{e}}la Nov{\'{a}}k, \emph{Sniffers,
  buzzers, toggles and blinkers: dynamics of regulatory and signaling pathways
  in the cell}, Current Opinion in Cell Biology \textbf{15} (2003), 221--231.

\bibitem{Walter2}
Wolfgang Walter, \emph{Differential and integral inequalities}, Translated from
  the German by Lisa Rosenblatt and Lawrence Shampine. Ergebnisse der
  Mathematik und ihrer Grenzgebiete, Band 55, Springer-Verlag, New York-Berlin,
  1970. \MR{0271508 (42 \#6391)}

\end{thebibliography}

\providecommand{\bysame}{\leavevmode\hbox to3em{\hrulefill}\thinspace}
\providecommand{\MR}{\relax\ifhmode\unskip\space\fi MR }
\providecommand{\MRhref}[2]{%
  \href{http://www.ams.org/mathscinet-getitem?mr=#1}{#2}
}
\providecommand{\href}[2]{#2}

\end{document}